\documentclass[twoside,12pt,a4paper]{amsart}
\usepackage{amscd,amsmath}
\usepackage{amssymb}
\usepackage{amsthm}
\usepackage{color}
\usepackage{hyperref}
\usepackage{enumerate}
\usepackage{geometry}
\theoremstyle{definition}
\newtheorem{theorem}{Theorem}[section]
\newtheorem{definition}[theorem]{Definition}

\newtheorem{lemma}[theorem]{Lemma}
\newtheorem{corollary}[theorem]{Corollary}
\newtheorem{remark}[theorem]{Remark}

\numberwithin{equation}{section}
\newtheorem{problem}{Problem}

\def\<{\left < }
\def\>{\right >}
\def\({\left ( }
\def\){\right )}
\def\e{\eqref}

\def\C2{${\bf C}^2$}

\begin{document}

\title[A General Inequality for Warped Product $CR$-Subman. of K$\ddot{A}$hler Man.]{A General Inequality for Warped Product $CR$-Submanifolds of K$\ddot{A}$hler Manifolds}

\author[A. Mustafa]{Abdulqader MUSTAFA}
\address{Department of Mathematics, Faculty of Arts and Science, Palestine Technical University, Kadoorei, Tulkarm, Palestine}
\email{abdulqader.mustafa@ptuk.edu.ps}.
\author[Cenap $\ddot{O}$zel]{Cenap $\ddot{O}$ZEL}
\address{Department of Mathematics, Faculty of Science, King Abdulaziz University, 21589 Jeddah, Saudi Arabia}
\email{cozel@kau.edu.sa}.
\author[P. Linker]{Patrick LINKER}
\address{Department of Materials Testing, University of Stuttgart, Stuttgart, Germany}
\email{Mrpatricklinker@gmail.com}.
\author[M. Sati]{Monika SATI}
\address{Department of Mathematics, HNBGU, SRT Campus Badshahithaul, Tehri Garhwal, Uttarakhand, India}
\email{monikasati123@gmail.com}.
\author[A. Pigazzini]{Alexander PIGAZZINI}
\address{Mathematical and Physical Science Foundation, 4200 Slagelse, Denmark}
\email{pigazzini@topositus.com}.

\maketitle
\begin{abstract}
In this paper, warped product contact $CR$-submanifolds in Sasakian, Kenmotsu and cosymplectic manifolds are shown to possess a geometric property; namely $\mathcal{D}_T$-minimal. Taking benefit from this property, an optimal general inequality for warped product contact $CR$-submanifolds is established in both Sasakian and Kenmotsu manifolds by means of the Gauss equation, we leave cosyplectic because it is an easy structure. Moreover, a rich geometry appears when the necessity and sufficiency are proved and discussed in the equality case. Applying this general inequality, the inequalities obtained by Munteanu are derived as particular cases, whereas the inequality obtained in \cite{A} is corrected. Up to now, the method used by Chen and Munteanu can not extended for general ambient manifolds, this is because many limitations in using Codazzi equation. Hence, Our method depends on the Gauss equation. The inequality is constructed to involve an intrinsic invariant (scalar curvature) controlled by an extrinsic one (the second fundamental form), which provides an answer for Problem \ref{prob3}. As further research directions, we have addressed a couple of open problems arose naturally during this work and depending on its results.\\



\end{abstract}

\sloppy

\section{Introduction}

The notion of warped products has been playing some important roles in the theory of general relativity as they have been providing the best mathematical models of our universe for now. Also recently new types of warped product manifolds have been introduced (see for example, \cite{obl}, \cite{seq}, \cite{qESeq} and \cite{pndp}). 

Extrinsic and intrinsic Riemannian invariants have vast applications in other fields of science. Classically, among extrinsic invariants, the shape operator and the squared mean curvature are the most important ones. Among the main intrinsic invariants, sectional, Ricci and scalar curvatures are the well-known ones. So, based on Nash's Theorem, our research programs is to search for control of
extrinsic quantities in relation to intrinsic quantities of Riemannian manifolds via Nash's
Theorem and to search for their applications \cite{2233ee}, \cite{55kk99}. Since it is an inevitable motivation, this was quite enough for Chen to address the following problem

\begin{problem}\label{prob3}
Establish simple relationships between the main extrinsic invariants and the main intrinsic invariants of a submanifold.
\end{problem}
 
Several famous results in differential geometry, such as isoperimetric inequality, Chern-Lashof's inequality, and Gauss-Bonnet's theorem among others, can be regarded as results in this respect. The current paper aims to continue this sequel of inequalities.

Combining special case inequalities in \cite{2211gg}, we also have
\begin{theorem}\label{792}
Let $M^n=N_T\times _fN_\perp$ be a $CR$-warped product submanifold in a complex space form $\tilde M^{2m}(c_{Ka})$. Then, we have the following 
$$\frac{1}{2}||h||^2\ge 2n_1n_2\frac{c_{Ka}}{4}+n_2 ||\nabla \ln f||^2-n_2~ \Delta (\ln f).$$
\end{theorem}

The current paper is organized to include eight sections. After the introduction, we present in section two, preliminaries, the basic definitions and formulas. In section three, we prove preparatory basic lemmas, which are necessary and useful for next sections. In the fourth section, it has been shown that warped product $CR$-submanifolds in Kahler and nearly Kahler manifolds possess a geometric property; namely $\mathcal{D}_T$-minimal submanifolds. Section five is devoted to present the statement and proof of the the main theorem in this article, here we consider warped product $CR$-submanifolds in complex space form to prove a general inequality involving the scalar curvature and the the squared norm of the second fundamental form. This inequality is derived using the Gauss equation, it generalizes all other inequalities which were derived by means of Codazzi equation. Moreover, it presents a new answer for Problem \ref{prob3}. Section six provides many geometric applications, part of them is obtaining the inequalities of \cite{2211gg}] as particular case inequalities from our main inequality. In the seventh section, we extend this inequality to generalized complex space form as an ambient manifold. In the final section,  we hypothesize two open problems arose naturally due to the results of this work.

\section{Preliminaries}
Let $\tilde{M}^m$ be a $C^\infty$ real $m$-dimensional manifold\footnotemark[\value{footnote}]\footnotetext{Throughout this work, we use the symbol \textasciitilde $~$for ambient manifolds, in order to be distinguished from the terminology of submanifolds.}. The {\it curvature tensor} $\tilde R$ of $\tilde \nabla$ is a tensor field of type $(1,~3)$ given by 
\begin{equation}\label{C1}
\tilde R(X, ~Y) Z =\tilde\nabla_X\tilde\nabla_YZ-\tilde\nabla_Y\tilde\nabla_XZ-\tilde\nabla_{[X,~ Y]}Z,
\end{equation}
and the $(0,~4)$ tensor field defined by 
\begin{equation}\label{D1}
\tilde R(X,~Y,~Z,~W)=\tilde g(\tilde R(X,~Y)Z,~W)
\end{equation}
is called the {\it Riemannian curvature tensor}, for any $X,~ Y,~ Z,~W\in \Gamma(T\tilde M^m)$. It is well-known that the Riemannian curvature tensor is a local isometry invariant.

If we choose two linearly independent tangent vectors $X,~Y\in T_x\tilde M^m$, then the {\it sectional curvature} of the $2$-plane $\pi$ spanned by $X$ and $Y$ is given in terms of the Riemannian curvature tensor $\tilde R$ by
\begin{equation}\label{E1}
\tilde K(X\wedge Y)=\frac{\tilde g(\tilde R(X,Y)Y, X)}{\tilde g(X,X) \tilde g(Y,Y)- (\tilde g(X,Y))^2}.
\end{equation}
In case that the $2$-plane $\pi$ is spanned by orthogonal unit vectors $X$ and $Y$ from the tangent space $T_x\tilde M^m,~x\in \tilde M^m$, the previous definition may be written as 
\begin{equation}\label{28}
\tilde K(\pi)=\tilde K_{\tilde M^m}(X\wedge Y)=\tilde g(\tilde R(X, ~Y)Y,~X).
\end{equation}

 Next, consider a local field of orthonormal frames $\{e_1, \cdots, e_m\}$ on $\tilde M^m$. 

In this context, we shall define another important Riemannian intrinsic invariant called the {\it scalar curvature} of $\tilde M^m$, and denoted by $\tilde\tau (T_x\tilde M^m)$, which, at some $x$ in $\tilde M^m$, is given by 
\begin{equation}\label{10}
\tilde\tau (T_x\tilde M^m)=\sum_{1\leq i\textless j\leq m}\tilde K_{ij},
\end{equation}
where $\tilde K_{ij}=\tilde K(e_i\wedge e_j)$. It is clear that, equation (\ref{10}) is congruent to 
\begin{equation}\label{}
2\tilde\tau (T_x\tilde M^m)=\sum_{1\leq i\neq j\leq m}\tilde K_{ij}.
\end{equation}

In particular, for a $2$-dimensional Riemannian manifold, the scalar curvature is its {\it Gaussian curvature}.

Next, we recall two important differential operators of a differentiable function $\psi$ on $\tilde M^m$; namely the {\it gradient} $\tilde\nabla \psi$ and the {\it Laplacian} $\Delta \psi$ of $\psi$, which are defined, respectively, as follows
\begin{equation}\label{L1}
\tilde g(\tilde\nabla \psi, X) = X\psi
\end{equation}
and
\begin{equation}\label{M1}
\Delta \psi =\sum_{i=1}^{m}((\tilde\nabla_{e_i}e_i)\psi- e_ie_i \psi), 
\end{equation}
for any vector field $X$ tangent to $\tilde M^m$, where $\tilde\nabla$ denotes the Levi-Civita connection on $\tilde M^m$. As a consequence, we have
\begin{equation}\label{consequence}
||\tilde \nabla \psi||^2=\sum_{i=1}^m\bigl(e_i(\psi)\bigr)^2.
\end{equation}

From the integration theory of manifolds, if $\tilde M^m$ is orientable compact, then we have
\begin{equation}\label{22}
\int_{\tilde M^m} \Delta f dV=0,
\end{equation}
where $dV$ denotes to the volume element of $\tilde M^m$.

In an attempt to construct manifolds of negative curvatures, in \cite{ddyy7} introduced the notion of {\it warped product manifolds} as follows:

\noindent
Let $N_1$ and $N_2$ be two Riemannian manifolds with Riemannian metrics $g_{N_1}$ and $g_{N_2}$, respectively, and $f>0$ a $C^\infty$ function on $N_1$. Consider the product manifold $N_1\times N_2$ with its projections $\pi_1:N_1\times N_2\mapsto N_1$ and $\pi_2:N_1\times N_2\mapsto N_2$. Then, the {\it warped product} $\tilde M^m= N_1\times _fN_2$ is the Riemannian manifold $N_1\times N_2=(N_1\times N_2, \tilde g)$ equipped with a Riemannian structure such that 
 $\tilde g=g_{N_1} + f^2 g_{N_2}$.

A warped product manifold $\tilde M^m=N_1\times _fN_2$ is said to be {\it trivial} if the warping function $f$ is constant. For a nontrivial warped product $N_1\times _fN_2$, we denote by $\mathcal{D}_1$ and $\mathcal{D}_2$ the distributions given by the vectors tangent to leaves and fibers, respectively. Thus, $\mathcal{D}_1$ is obtained from tangent vectors of $N_1$ via the horizontal lift and $\mathcal{D}_2$ is obtained by tangent vectors of $N_2$ via the vertical lift.

Now, let $\{e_1,\cdots,e_{n_1}, e_{n_1+1}, \cdots, e_m\}$ be local fields of orthonormal frame of $\Gamma (T\tilde M^m)$ such that $n_1$, $n_2$ and $m$ are the dimensions of $N_1$, $N_2$ and $\tilde M^m$, respectively. Then, for any Riemannian warped product $ \tilde M^m=N_1\times _fN_2$. It is well known that the sectional curvature and the warping function are related by [\cite{2233ee}, \cite{55kk99}, \cite{ssee44}]
\begin{equation}\label{24}
\sum_{a=1}^{n_1}\sum_{A=n_1+1}^{m} \tilde K(e_a\wedge e_A)=\frac{n_2\Delta f}{f}.
\end{equation}

Now, we turn our attention to the differential geometry of the submanifold theory. The {\it Gauss} and {\it Weingarten formulas} are, respectively, given by
\begin{equation}\label{3}
\tilde \nabla_X Y=\nabla_X Y+h(X,Y) 
\end{equation}
and
\begin{equation}\label{4}
\tilde\nabla_X\zeta=-A_\zeta X+\nabla^\perp_X\zeta 
\end{equation}
for all $X,Y\in \Gamma(TM^n)$ and $\zeta\in \Gamma (T^\perp M^n)$, where $\tilde\nabla$ and $\nabla$ denote respectively for the Levi-Civita and the {\it induced} Levi-Civita connections on $\tilde M^m$ and $M^n$, and $\Gamma(TM^n)$ is the module of differentiable sections of the vector bundle $TM^n$. $\nabla^\perp$ is the {\it normal connection} acting on the normal bundle $T^\perp M^n$. 

Here, it is well-known that the {\it second fundamental form} $h$ and the {\it shape operator} $A_\zeta$ of $M^n$ are related by  
\begin{equation}\label{5}
g(A_\zeta X,Y)=g(h(X,Y),\zeta)
\end{equation}
for all $X,Y\in \Gamma(TM^n)$ and $\zeta\in \Gamma(T^\perp M^n)$, [\cite{pom}, \cite{iijj77}]. 

Geometrically, $M^n$ is called a {\it totally geodesic} submanifold in $\tilde M^m$ if $h$ vanishes identically. Particularly, the {\it relative null space}, ${\mathcal N}_x$, of the submanifold $M^n$ in the Riemannian manifold $\tilde M^m$ is defined at a point $x\in M^n$ by as
\begin{equation}\label{19}
{\mathcal N}_x=\{ X\in T_xM^n: h(X, Y)=0~~~ \forall~ Y\in T_xM^n\}.
\end{equation}

Likewise, we consider a local field of orthonormal frames \footnotemark[\value{footnote}] \footnotetext{Throughout this work, $M^n=N_1\times _fN_2$ denotes for the isometrically immersed warped product submanifold in $\tilde M^m$. The numbers $m,~n,~n_1,$ and $n_2$ are the dimensions of $\tilde M^m$, $M^n$, $N_1$ and $N_2$, respectively.} 
$\{e_1, \cdots , e_n, e_{n+1}, \cdots, e_m\}$ on $\tilde M^m$, such that, restricted to $M^n$, $\{e_1, \cdots , e_n\}$ are tangent to $M^n$ and $\{e_{n+1}, \cdots, e_m\}$ are normal to $M^n$. Then, the {\it mean curvature vector} $\vec H(x)$ is introduced as 
\begin{equation}\label{8}
\vec H(x)=\frac{1}{n} \sum_{i=1}^{n} h(e_i, e_i),
\end{equation}

On one hand, we say that $M^n$ is a {\it minimal submanifold} of $\tilde M^m$ if $\vec H=0$. On the other hand, one may deduce that $M^n$ is totally umbilical in $\tilde M^m$ if and only if $h(X,Y)=g(X,Y) \vec H$, for any $X,~Y\in \Gamma (TM^n)$. It is remarkable to note that the scalar curvature $\tau (x)$ of $M^n$ at $x$ is identical with the scalar curvature of the tangent space $T_xM^n$ of $M^n$ at $x$; that is, $\tau (x)= \tau (T_xM^n)$ [\cite{2233ee}].

In this series, the well-known {\it equation} of {\it Gauss} is given by 
\begin{equation}\begin{split}\label{12}
R(X,Y,Z,W)= \tilde R(X,Y,Z,W)
\\
+ g(h(X, W), h(Y, Z)) - g(h(X, Z), h(Y, W)),
\end{split}
\end{equation}

\noindent
for any vectors $X, ~Y,~ Z,~ W\in \Gamma (TM^n)$, where $\tilde R$ and $R$ are the curvature tensors of $\tilde M^m$ and $M^n$, respectively.

From now on, we refer to the coefficients of the second fundamental form $h$ of $M^n$ with respect to the above local frame by the following notation 
\begin{equation}\label{13}
h_{ij}^r=g(h(e_i,e_j),e_r),
\end{equation}
 where $i, j \in \{1, . . . , n\}$, and $r \in \{ n+1, . . . , m\}$. First, by making use of $(\ref{13})$, $(\ref{12})$ and $(\ref{28})$, we get the following
\begin{equation}\label{}
 K (e_i \wedge e_j )= \tilde K (e_i \wedge e_j) + \sum_{r=n+1}^{m} (g(h_{ii}^{r}~e_r, h_{jj}^{r} ~e_r)- g(h_{ij}^{r}~e_r, h_{ij}^{r}~e_r)).
\end{equation}

Equivalently,
\begin{equation}\label{14}
 K (e_i \wedge e_j )= \tilde K (e_i \wedge e_j) + \sum_{r=n+1}^{m} (h_{ii}^{r} h_{jj}^{r}- (h_{ij}^{r})^2),
\end{equation}
where $\tilde K (e_i \wedge e_j)$ denotes the sectional curvature of the $2$-plane spanned by $e_i$ and $e_j$ at $x$ in the ambient manifold $\tilde M^m$. Secondly, by taking the summation in the above equation over the orthonormal frame of the tangent space of $M^n$, and due to $(\ref{10})$, we immediately obtain
\begin{equation}\label{15}
2\tau (T_xM^n)= 2\tilde\tau (T_xM^n) +n^2 ||\vec H||^2-||h||^2,
\end{equation}
where
\begin{equation}\label{16}
 \tilde\tau (T_xM^n)=\sum_{1\le i< j\le n}\tilde K (e_i \wedge e_j)
\end{equation}
denotes the scalar curvature of the n-plane $T_xM^n$ in the ambient manifold $\tilde M^m$.

For a warped product $M^n= N_1\times _fN_2$, let $\varphi : M^n\rightarrow \tilde M^m$ be an isometric immersion of $N_1\times _fN_2$ into an arbitrary Riemannian manifold $\tilde M^m$. As usual, let $h$ be the second fundamental form of $\varphi$. We call the immersion $\varphi$ {\it mixed totally geodesic} if $h(X,Z)=0$ for any $X$ in $\mathcal{D}_1$ and $Z$ in $\mathcal{D}_2$, \cite{2233ee}. In particular, if we denote the restrictions of $h$ to $N_1$ and $N_2$ respectively by $h_1$ and $h_2$, then for $i=1$ and $2$, we call $h_i$ the {\it partial second fundamental form} of $\varphi$. Automatically, the {\it partial mean curvature vectors} $\vec H_1$ and $\vec H_2$ are defined by the following partial traces \footnotemark[\value{footnote}]\footnotetext{Throughout this work, we use the following convention on the range of indices unless otherwise stated, the indices $i,j$ run from 1 to $n$, the lowercase letters $a,b$ from 1 to $n_1$, the uppercase letters $A,B$ from $n_1$ to $n$ and $r$ from $n$ to $m$.}
\begin{equation}\label{N1}
\vec H_1=\frac{1}{n_1}\sum_{a=1}^{n_1}h(e_a,e_a),~~~\vec H_2=\frac{1}{n_2}\sum_{A=n_1+1}^{n_1+n_2}h(e_A,e_A)
\end{equation}
for some orthonormal frame fields $\{e_1,\cdots, e_{n_1}\}$ and $\{e_{n_1+1},\cdots, e_{n_1+n_2}\}$ of $\mathcal{D}_1$ and $\mathcal{D}_2$, respectively. 

This motivation for the following definition may not be evident at this moment, but it will emerge gradually as we prove its natural existence, then imposing it to have profoundly general results, [\cite{pom}, \cite{2233ee}, \cite{yyhh88}, \cite{kkxx77}, \cite{11bbyy}, \cite{aassll}].
\begin{definition}\label{9}
An immersion $\varphi : N_1\times _fN_2\longrightarrow \tilde M^m$ is called $\mathcal{D}_i$-totally geodesic if the partial second fundamental form $h_i$ vanishes identically. If for all $X,~Y\in \mathcal{D}_i$ we have $h(X,~Y)=g(X,~Y){\mathcal K}$ for some normal vector ${\mathcal K}$, then $\varphi$ is called $\mathcal{D}_i$-totally umbilical. It is called $\mathcal{D}_i$-minimal if the partial mean curvature vector $\vec H_i$ vanishes, for $i=1$ or 2.
\end{definition}

For an odd dimensional real $C^\infty$ manifold $\tilde M^{2l+1}$, let $\phi$, $\xi$, $\eta$ and $\tilde g$ be respectively a $(1,~1)$ tensor field, a vector field, a $1$-form and a Riemannian metric on $\tilde M^{2l+1}$ satisfying
\begin{equation}\label{151}
\left.
\begin{aligned}
   \phi^2=-I+\eta\otimes\xi,~~~~~\phi\xi=0,~~~~~\eta\circ\phi=0,~~~~~\eta(\xi) = 1\\ 
   \eta(X)=\tilde g(X, \xi),~~~~~\tilde g(\phi X, \phi Y)=\tilde g(X, Y)-\eta(X)\eta(Y),
\end{aligned}
\right\}
\end{equation}
for any $X,~Y\in \Gamma (T\tilde M^{2l+1})$. Then we call $(\tilde M^{2l+1}, \phi, \xi, \eta, \tilde g)$ an {\it almost contact metric manifold} and $(\phi, \xi, \eta, \tilde g)$ an {\it almost contact metric structure} on $\tilde M^{2l+1}$, [\cite{K},~\cite{B}].

A fundamental $2$-form $\Phi$ is defined on $\tilde M^{2l+1}$ by $\Phi(X,Y)=\tilde g(\phi X,Y)$. An almost contact metric manifold $\tilde M^{2l+1}$ is called a contact metric manifold if $\Phi =\frac{1}{2} d\eta$. If the almost contact metric manifold $(\tilde M^{2l+1},\phi, \xi, \eta, \tilde g)$ satisfies $[\phi, \phi]+2d\eta\otimes \xi=0$, then $(\tilde M^{2l+1},\phi, \xi, \eta, \tilde g)$ turns out to be a {\it normal almost contact manifold}, where the Nijenhuis tensor is defined as
$$[\phi,\phi](X,Y)=[\phi X, \phi Y]+\phi^2[X,Y]-\phi [X,\phi Y]-\phi[\phi X, Y]~~~~~\forall~ X,~Y\in \Gamma(T\tilde M^{2l+1}).$$

For our purpose, we will distinguish four classes of almost contact metric structures; namely, Sasakian, Kenmotsu, cosymplectic and nearly trans-Sasakian structures. At first, an almost contact metric structure is  is said to be {\it Sasakian} whenever it is both contact metric and normal, equivalently [\cite{S}]
\begin{equation}\label{157}
(\tilde\nabla_X\phi)Y=-\tilde g(X,Y)\xi+\eta(Y)X.
\end{equation}

A $2$-plane $\pi$ in $T_x\tilde M^{2l+1}$ of an almost metric manifold $\tilde M^{2l+1}$ is called a $\phi$-section if $\pi\perp\xi$ and $\phi (\pi)=\pi$. Accordingly, we say that $\tilde M^{2l+1}$ is of constant $\phi$-sectional curvature if the sectional curvature $\tilde K(\pi)$ does not depend on the choice of the $\phi$-section $\pi$ of $T_x\tilde M^{2l+1}$ and the choice of a point $x\in \tilde M^{2l+1}$. Based on this preparatory concept, a Sasakian manifold $\tilde M^{2l+1}$ is said to be a {\it Sasakian space form} $\tilde M^{2l+1}(c_S)$, if the $\phi$-sectional curvature is constant $c_S$ along $\tilde M^{2l+1}$. Then the associated Riemannian curvature tensor $\tilde R$ on $\tilde M^{2l+1}(c_S)$ is given by [\cite{K}]
\begin{equation}\label{158}
\tilde R(X,Y;Z,W)=\frac{c_S+3}{4}\biggl\{\tilde g(X,W)\tilde g(Y,Z)-\tilde g(X,Z)\tilde g(Y,W)\biggr\}$$$$-\frac{c_S-1}{4}\biggl\{ \eta (Z) \biggl(\eta (Y) \tilde g(X,W) - \eta (X) \tilde g(Y,W)\biggr) $$$$+ \biggl( \tilde g(Y,Z) \eta (X) - \tilde g(X,Z) \eta (Y)\biggr) \tilde g(\xi , W) $$$$-\tilde g(\phi X,W)\tilde g(\phi Y,Z)+ \tilde g(\phi X,Z)\tilde g(\phi Y,W)+2\tilde g(\phi X,Y)\tilde g(\phi Z,W)\biggr\},
\end{equation}
for any $X,~Y,~Z,~W\in \Gamma (T\tilde M^{2l+1}(c_S)).$ 

 An almost contact metric manifold $\tilde M^{2l+1}$ is called {\it Kenmotsu manifold} [\cite{K}] if 
\begin{equation}\label{155}
(\tilde\nabla_X\phi)Y=\tilde g(\phi X, Y)\xi -\eta(Y)\phi X,
\end{equation}

By analogy with Sasakian manifolds, a Kenmotsu manifold $\tilde M^{2l+1}$ is said to be a {\it Kenmotsu space form} $\tilde M^{2l+1}(c_{Ke})$, if the $\phi$-sectional curvature is constant $c_{Ke}$ along $\tilde M^{2l+1}$, whose Riemannian curvature tensor $\tilde R$ on $\tilde M^{2l+1}(c_{Ke})$ is characterized by [\cite{A}]
\begin{equation}\label{156}
\tilde R(X,Y;Z,W)=\frac{c_{Ke}-3}{4}\biggl\{\tilde g(X,W)\tilde g(Y,Z)-\tilde g(X,Z)\tilde g(Y,W)\biggr\}$$$$-\frac{c_{Ke}+1}{4}\biggl\{ \eta (Z) \biggl(\eta (Y) \tilde g(X,W) - \eta (X) \tilde g(Y,W)\biggr) $$$$+ \biggl( \tilde g(Y,Z) \eta (X) - \tilde g(X,Z) \eta (Y) \biggr) \tilde g(\xi , W) $$$$-\tilde g(\phi X,W)\tilde g(\phi Y,Z)+ \tilde g(\phi X,Z)\tilde g(\phi Y,W)+2\tilde g(\phi X,Y)\tilde g(\phi Z,W)\biggr\},
\end{equation}
for any $X,~Y,~Z,~W\in\Gamma(T\tilde M^{2l+1}(c_{Ke})).$ We notice that Kenmotsu manifolds are normal but not quasi-Sasakian and hence not Sasakian [\cite{B}].

In the case of killing almost contact structure tensors, consider a normal almost contact metric structure $(\phi, \xi,\eta,\tilde g)$ with both $\Phi$ and $\eta$ are closed. Then, such $(\phi, \xi,\eta,\tilde g)$ is called {\it cosymplectic}. Explicitly, cosymplectic manifolds are characterized by normality and the vanishing of Riemannian covariant derivative of $\phi$, i.e.,
\begin{equation}\label{152}
(\tilde\nabla_X\phi)Y=0.
\end{equation}

A cosymplectic manifold $\tilde M^{2l+1}$ is said to be a {\it cosymplectic space form} $\tilde M^{2l+1}(c_c)$, if the $\phi$-sectional curvature is constant $c_c$ along $\tilde M^{2l+1}$ with Riemannian curvature tensor $\tilde R$ expressed by [\cite{B}]
\begin{equation}\label{153}
\tilde R(X,Y;Z,W)=\frac{c_c}{4}\biggl\{\tilde g(X,W)\tilde g(Y,Z)-\tilde g(X,Z)\tilde g(Y,W)$$$$- \eta (Z) \biggl(\eta (Y) \tilde g(X,W) - \eta (X) \tilde g(Y,W)\biggr) - \biggl( \tilde g(Y,Z) \eta (X) - \tilde g(X,Z) \eta (Y) \biggr) \tilde g(\xi , W) $$$$+\tilde g(\phi X,W)\tilde g(\phi Y,Z)- \tilde g(\phi X,Z)\tilde g(\phi Y,W)-2\tilde g(\phi X,Y)\tilde g(\phi Z,W)\biggr\},
\end{equation}
for any $X,~Y,~Z,~W\in \Gamma (T\tilde M^{2l+1}(c_c)).$ Hereafter, we call the almost contact manifold $\tilde M^{2l+1}$ a {\it nearly cosymplectic } manifold if 
\begin{equation}\label{154}
(\tilde \nabla_X\phi)Y+(\tilde \nabla_Y\phi)X=0.
\end{equation}

A submanifold $M^n$ of an almost contact metric manifold $\tilde M^{2l+1}$ is said to be a {\it contact CR-submanifold } if there exist on $M^n$ differentiable distributions $\mathcal{D}_T$ and $\mathcal{D}_\perp$, satisfying the following
\begin{enumerate}
\item [(i)] $TM^n=\mathcal{D}_T\oplus \mathcal{D}_\perp\oplus \langle\xi\rangle$,
\item [(ii)] $\mathcal{D}_T$ is an invariant distribution, i.e., $ \phi (\mathcal{D}_T) \subseteq \mathcal{D}_T$,
\item [(iii)] $\mathcal{D}_\perp$ is an anti-invariant distribution, i.e., $\phi (\mathcal{D}_\perp) \subseteq T^\perp M^n$.
\end{enumerate}

Denote by $\nu$ the maximal $\phi$-invariant subbundle of the normal bundle $T^\perp M^n$. Then it is well-known that the normal bundle $T^\perp M^n$ admits the following decomposition 
\begin{equation}\label{60}
T^\perp M^n=F\mathcal{D}_\perp \oplus \nu.
\end{equation}

In almost contact manifolds $\tilde M^{2m+1}$, the warped product $N_T\times _fN_\perp$ is called a {\it $CR$-warped product} submanifold, if the submanifolds $N_T$ and $N_\perp$ are integral manifolds of $\mathcal{D}_T$ and $\mathcal{D}_\perp$, respectively.

\section{Basic Lemmas}

Now, we turn our attention to almost contact manifolds, we are going to explain the natural existence of $\mathcal{D}_i$-minimal warped product submanifolds in almost contact manifolds, for both $i=1$ and $i=2$. Observe that all almost contact manifolds considered in this thesis satisfy $(\tilde\nabla_\xi \phi)\xi=0$. Hence, it is convenient to state 
\begin{lemma}\label{1551}
Let $M^n$ be a submanifold tangent to the characteristic vector field $\xi$ in an almost contact manifold $\tilde M^{2l+1}$. If $(\tilde\nabla_\xi \phi)\xi=0$ on $\tilde M^{2l+1}$, then $h(\xi,\xi)=0$.
\end{lemma}
Beginning with Sasakian manifolds, we call a warped product of type $M^n=N_T\times _fN_\perp$, a contact CR- warped product submanifold. 
\begin{corollary}\label{186}
Let $M^n=N_T\times _fN_\perp$ be a contact CR- warped product submanifold in a Sasakian manifold $\tilde {M}^{2l+1}$ such that $\xi$ is tangent to the first factor. Then, the following hold
\begin {enumerate}
\item [(i)] $h(X,\xi)=0$;
\item [(ii)] $g(h(X,X), FZ)=0$;
\item [(iii)] $g(h(X,X),\zeta)=-g(h(\phi X, \phi X), \zeta)$,
\end {enumerate}
for every $X\in \Gamma(TN_T)$, $Z\in \Gamma(TN_\perp)$ and $\zeta \in \Gamma(\nu)$. 
\end{corollary}
\begin{proof}
From $(\ref{157})$ we obtain
$$X-\eta(X)\xi=-\phi\nabla_X\xi-\phi h(X,\xi).$$
Applying $\phi$ on the above equation, taking into consideration $\eta(\nabla_X\xi)=0$, then it yields
$$\phi X=\nabla_X\xi+ h(X,\xi).$$
By comparing the tangential and normal terms in the above equation we get $(i)$. $(ii)$ is well-known (for example, see [\cite{Mi}], [\cite{M}]). For the last part, we take an arbitrary $\zeta \in \Gamma(\nu)$, then by making use of (\ref{157}) and (\ref{3}), we obtain
$$\nabla_X\phi X+h(\phi X,X)-\phi \nabla_XX-\phi h(X,X)=-g(X,X)\xi + \eta (X)X,$$
taking the inner product with $\phi \zeta$ in the above equation, we deduce 
\begin{equation}\label{1028}
g(h(\phi X,X), \phi \zeta)-g(h(X,X), \zeta)=0,
\end{equation}
interchanging $X$ with $\phi X$ in (\ref{1028}), gives
$$g(h(\phi X,\phi X), \zeta)=g(h(\phi (\phi X), \phi X), \phi \zeta)= g(\tilde \nabla_{\phi X}\phi (\phi X), \phi \zeta)$$$$~~~~~~~~~~~~~~~= -g(\tilde\nabla_{\phi X} X, \phi \zeta)+g(\tilde\nabla_{\phi X} (\eta(X)\xi), \phi \zeta)$$$$~~~~~~~~~~~~~~~~~= - g(h(X, \phi X), \phi \zeta)+\eta(X)g(\tilde\nabla_{\phi X} \xi, \phi \zeta)
$$$$~~~~~~~~~~~~~~~~~~~~~= - g(h(X, \phi X), \phi \zeta)+\eta(X)g(h(\phi X, \xi), \phi \zeta).$$
Making use of statement $(i)$ in the above equation, we reach that
\begin{equation}\label{-1}
g(h(\phi X,\phi X), \zeta)= - g(h(X, \phi X), \phi \zeta).
\end{equation}
From $(\ref{1028})$ and $(\ref{-1})$, we obtain statement $(iii)$.
\end{proof}

The following two direct, but significant, results are two other key lemma for this section that will be used later as well. 

\begin{lemma}\label{firstlemma}
Let $\varphi :M^n=N_1\times _fN_2 \longrightarrow \tilde M^m$ be an isometric immersion of an $n$-dimensional warped product submanifold $M^n$ into a Riemannian manifold $\tilde M^m$. Then, we have
\begin{equation}\label{85}
\tau \bigl(T_xM^n\bigr)=\frac{n_2 \Delta f}{f}+\sum_{r=n+1}^{m}\biggl\{ \sum_{1\le a <b \le n_1} \biggl(h_{aa}^r h_{bb }^r-\bigl(h_{a b}^r\bigr)^2\biggr)~~~~~~~~~~~~~~~~~$$$$~~~~~~~~~~~~~~~~~~~~~~~~~~~~~~~~~~~~~~~~~+ \sum_{n_1+1\le A <B\le n}\biggl(h_{AA}^rh_{BB}^r-\bigl(h_{AB}^r\bigr)^2\biggr)\biggr\}
+~\tilde\tau \bigl(T_xN_1\bigr)+\tilde \tau \bigl(T_xN_2\bigr),~~~
\end{equation}
where $n_1,~n_2, ~n$ and $m$ are the dimensions of $N_1,~N_2, ~M^n$ and $\tilde M^m$, respectively.
\end{lemma}
\begin{proof}
From the definition of the scalar curvature, we have 
\begin{equation}\label{}
\tau \bigl(T_xM^n\bigr)=\sum_{1\le i<j\le n} K_{ij}=\sum_{a=1}^{n_1} \sum_{A=n_1+1}^{n}K_{aA}+\sum_{1\le a<b\le n_1} K_{ab}+\sum_{n_1+1\le A<B\le n} K_{AB}.
\end{equation}
Now, we recall the following well-known relation
\begin{equation}\label{24}
\sum_{a=1}^{n_1}\sum_{A=n_1+1}^{n} K(e_a\wedge e_A)=\frac{n_2\Delta f}{f},
\end{equation}
where $\{e_1,\cdots,e_{n_1}, e_{n_1+1}, \cdots, e_n\}$ are local fields of orthonormal frame of $\Gamma (TM^n)$ such that $n_1$, $n_2$ and $n$ are the dimensions of $N_1$, $N_2$ and $M^n$, respectively.
 Combining the above two equations, it yields
\begin{equation}\label{jfle}
\tau \bigl(T_xM^n\bigr)=\frac{n_2 \Delta f}{f}
+\tau \bigl(T_xN_1\bigr)+ \tau \bigl(T_xN_2\bigr).
\end{equation}
It is direct to write
\begin{equation}\label{,qod}
\tau \bigl(T_xN_1\bigr)=\sum_{r=n+1}^{m} \sum_{1\le a <b \le n_1} \biggl(h_{aa}^r h_{bb }^r-\bigl(h_{a b}^r\bigr)^2\biggr)+\tilde\tau \bigl(T_xN_1\bigr),
\end{equation}
and
\begin{equation}\label{migjk}
\tau \bigl(T_xN_2\bigr)=\sum_{r=n+1}^{m} \sum_{n_1+1\le A <B\le n}\biggl(h_{AA}^rh_{BB}^r-\bigl(h_{AB}^r\bigr)^2\biggr)+\tilde \tau \bigl(T_xN_2\bigr).
\end{equation}
By joining $(\ref{jfle})$, $(\ref{,qod})$ and $(\ref{migjk})$ together, we get the result.
\end{proof}

\begin{lemma}\label{two}
Let $\varphi$ be a $\mathfrak{D}_2$-minimal isometric immersion of a warped product submanifold $M^n=N_1\times _fN_2$ into any Riemannian manifold $\tilde {M}^m$. If $N_2$ is totally umbilical in $\tilde M^m$, then $\varphi$ is $\mathfrak{D}_2$-totally geodesic.
\end{lemma}
\begin{proof}
Let $\check{h}$ and $\hat {h}$ denote the second fundamental forms of $N_2$ in $M^n$ and $\tilde M^m$, respectively. Then for every vector fields $Z$ and $W$ tangent to $N_2$ we have
\begin{equation}
h(Z,W)=\hat {h}(Z,W)-\check{h}(Z,W),
\end{equation}
and 
\begin{equation}
\check{ h}(Z,W)= - \bigl(g(Z, W)/f\bigr) \nabla (f).
\end{equation}
Notice that, for every warped product the leaves are totally geodesic and the fibers are totally umbilical. Taking in consideration this fact and our  hypothesis guarantees that $N_2$ is totally umbilical in both $M^n$ and $\tilde M^m$. Considering this fact with the above two equations, we deduce that
\begin{equation}\label{lsi}
h(Z,W)=g(Z,W) (\Psi+\nabla (\ln f)), 
\end{equation}
for some vector field $\Psi\in\Gamma(T\tilde M^m)$ such that $\Psi$ is normal to $\Gamma(TN_2)$.
Considering the local field of orthonormal frames as in the above proof. Then, taking the summation over the orthonormal frame fields of $\Gamma (TN_2)$ in the above equation, we get
$$\sum_{A,B=n_1+1}^{n}h(e_A,e_B)=\sum_{A,B=n_1+1}^{n}g(e_A,e_B) (\Psi+\nabla (\ln f)).$$
Taking into account $\mathfrak{D}_2$-minimality of $\varphi$, the left hand side of the above equation vanishes and we get
$$0=n_2 ~(\Psi+\nabla (\ln f)).$$
Since $N_2$ is not empty, we obtain
$$\Psi =-\nabla(\ln f).$$
Making use of the above equation in (\ref{lsi}), we obtain
$$h(Z,W)=0,$$
for every vector fields $Z, W\in \Gamma(TN_2).$ Meaning that; $\varphi$ is $\mathfrak{D}_2$-totally geodesic. This completes the proof.
\end{proof}

\section{$\mathcal{D}_T$-Minimality of Warped Product $CR$-Submanifolds in Kahler Manifolds}

Recently, it was proven that $\mathcal{D}_T$-minimality is possessed by a wide class of warped product submanifolds, some of these warped product submanifolds were shown to have this geometric property in \cite{11bbyy}, \cite{aassll}.

In the sense of Definition \ref{9}, we are going to show the natural existence of $\mathcal{D}_T$-minimal warped product $CR$-submanifolds in both Kahler and nearly Kahler manifolds.

Secondly, we provide the next key result which will be referred to frequently during this section.
\begin{lemma}\label{1B} 
Let $M^n=N_T\times _fN_\perp$ be a contact $CR$-warped product submanifold in Sasakian manifolds $\tilde M^{2l+1}$ such that $\xi$ is tangent to $N_T$. Then, $M^n$ is $\mathcal{D}_1$-minimal warped product, where $\mathcal{D}_1=\mathcal{D}_T\oplus\langle\xi\rangle$.
\end{lemma}

\begin{proof}
Consider the following local field of orthonormal frames of the Kahler manifold $\tilde M^{2m}$:
$\{\xi, e_1, \cdots, e_s, e_{s+1}=\phi e_1, \cdots, e_{n_1}=e_{2s}=\phi e_s, e_{n_1+1}=e^\star_1,\cdots, e_{n_1+n_2}=e_n=e^\star_q, e_{n+1}=\phi e^\star_1, \cdots, e_{n+n_2}=\phi e^\star_q, e_{n+n_2+1}=\bar e_1, \cdots, e_{2m}=\bar e_{2l=\gamma}\}$ such that $\{e_1, \cdots, \\e_s, e_{s+1}=\phi e_1, \cdots, e_{n_1}=e_{2s}=\phi e_s\}$, $\{e_{n_1+1}=e^\star_1,\cdots, e_{n_1+n_2}=e_n=e^\star_q\}$, $\{e_{n+1}=\phi e^\star_1, \cdots, e_{n+n_2}=\phi e^\star_q\}$ and $\{e_{n+n_2+1}=\bar e_1, \cdots, e_{2m}=\bar e_{2l=\gamma}\}$ are the local fields of orthonormal frames of $\Gamma(TN_T)$, $\Gamma (TN_\perp)$, $\Gamma (JTN_\perp)$ and $\Gamma (\nu)$, respectively.

Using the terminology in (\ref{13}), it is straightforward to have
$$\sum_{r=n+1}^{2m}\sum_{a=1}^{n_1}h_{aa}^r= \sum_{r=n+1}^{2m}\bigl(h_{11}^r+ \cdots +h_{n_1n_1}^r \bigr).$$

In view of (\ref{60}), the right hand side summation can be decomposed as
$$\sum_{r=n+1}^{2m}\sum_{a=1}^{n_1}h_{aa}^r= \sum_{r=n+1}^{2m-\gamma}\bigl(h_{11}^r+ \cdots +h_{n_1n_1}^r \bigr)+ \sum_{r=n+1+q}^{2m}\bigl(h_{11}^r+ \cdots +h_{n_1n_1}^r \bigr).$$

Taking into account part $(i)$ of Corollary \ref{186}, the first summation on the right hand side of the above equation vanishes, whereas we expand the other summation in view of the above orthonormal frames to get
$$\sum_{r=n+1}^{2m}\sum_{a=1}^{n_1}h_{aa}^r= \sum_{r=n+1+q}^{2m}\bigl(h_{11}^r+ \cdots +h_{ss}^r+h_{s+1s+1}^r+ \cdots + h_{2s2s}^r \bigr).$$

Equivalently,
$$\sum_{r=n+1}^{2m}\sum_{a=1}^{n_1}h_{aa}^r= \sum_{r=n+1+q}^{2m}\biggl( g(h(e_1,e_1),e_r)+ \cdots +g(h(e_s, e_s), e_r)$$$$~~~~~~~~~~~~~~~~~~~~~~~~~~~~~~~+g(h(J e_1, J e_1), e_r)+ \cdots + g(h(J e_s, J e_s), e_r)\biggr).$$

Now, if we apply part $(ii)$ of Corollary \ref{186} on the above equation, then it automatically gives 
$$\sum_{r=n+1}^{2m}\sum_{a=1}^{n_1}h_{aa}^r= \sum_{r=n+1+q}^{2m} \biggl(g(h(e_1,e_1),e_r)+ \cdots +g(h(e_s, e_s), e_r)$$$$~~~~~~~~~~~~~~~~~~~~~~~~-g(h(e_1,  e_1), e_r)- \cdots - g(h( e_s,  e_s), e_r)\biggr)$$$$=0.~~~~~~~~~~~~~~~~~~~~~~~~~
~~~~~~~~~~~~~~~~~~~~~~~~$$

Clearly, this proves the vanishing of the coefficients $h_{aa}^r$ under summation, for $a\in \{1,\cdots, n_1\}$ and $r\in \{n+1, \cdots, 2m\}.$ Therefore, the partial mean curvature vector $\vec H$ defined in (\ref{N1}) does vanish. Hence, in the sense of Definition \ref{9}, we get the assertion.
\end{proof}

 \begin{remark}\label{104}
Putting $\mathcal{D}_1=\mathcal{D}_T$, then by following the above scheme typically one can show that warped product submanifolds of the type $M^n=N_T\times _fN_\perp$, are $\mathcal{D}_1$-minimal in nearly Kahelr manifolds.
\end{remark}

\section{A General Inequality for Warped Product $CR$-Submanifolds in Kahler Manifolds}

By making use of the Gauss equation, we construct a new general inequality for $\mathcal{D}_T$-minimal warped product $CR$-submanifolds in arbitrary Kahler manifolds. This inequality generalizes all inequalities in [\cite{2211gg}].

Now, we present the main theorem of this article. 
\begin{theorem}\label{299}
Let $\varphi :M^n=N_T\times _fN_\perp \longrightarrow \tilde M^m$ be an isometric immersion of a warped product $CR$-submanifold $M^n$ into a Kahler manifold $\tilde {M}^m$. Then, we have
 \begin{enumerate}
\item[(i)]$\frac{1}{2}||h||^2\ge \tilde \tau (T_xM^n)-\tilde \tau (T_xN_T)-\tilde \tau (T_xN_\perp)-\frac {n_2 \Delta f}{f}.$
\item[(ii)] The equality in (i) holds identically if and only if $N_T$, $N_\perp$ and $M^n$ are totally geodesic, totally umbilical and minimal submanifolds in $\tilde M^m$, respectively.
\end{enumerate}
\end{theorem}

\begin{proof} 
Via (\ref{15}), we first have 
$$||h||^2= -2 \tau (T_xM^n)+2 \tilde \tau (T_xM^n)+ n^2||\vec H||^2.$$

In view of Lemma \ref{firstlemma}, the above equation takes the following form
$$||h||^2= 2 \tilde \tau (T_xM^n) - 2\tilde \tau (T_x N_T)-2 \tilde\tau (T_x N_\perp)-2 \frac {n_2 \Delta f}{f}+ n^2||\vec H||^2$$$$-2\left(\sum_{r=n+1}^{m}\sum_{1\le a<b\le n_1} \bigl(h_{aa}^r h_{bb}^r- (h_{ab}^r)^2\bigr)\right)$$$$~~~~~~~-2\left(\sum_{r=n+1}^{m}\sum_{n_1+1\le A<B\le n} \bigl(h_{AA}^r h_{BB}^r- (h_{AB}^r)^2\bigr)\right).$$
This is equivalent to 
\begin{equation}\label{1072}
||h||^2= 2 \tilde \tau (T_xM^n) - 2\tilde \tau (T_x N_T)-2 \tilde\tau (T_x N_\perp)-2 \frac {n_2 \Delta f}{f}+ n^2||\vec H||^2$$$$-\left(\sum_{r=n+1}^{m}\sum_{1\le a\ne b\le n_1} \bigl(h_{aa}^r h_{bb}^r- (h_{ab}^r)^2\bigr)\right)$$$$~~~~~~~-\left(\sum_{r=n+1}^{m}\sum_{n_1+1\le A\ne B\le n} \bigl(h_{AA}^r h_{BB}^r- (h_{AB}^r)^2\bigr)\right).
\end{equation}

\noindent
Since $\varphi$ is $\mathcal{D}_T$-minimal immersion, then
 $$-\left(\sum_{r=n+1}^{m}\sum_{1\le a\ne b\le n_1} \bigl(h_{aa}^r h_{bb}^r- (h_{ab}^r)^2\bigr)\right)=$$$$\sum_{r=n+1}^{m}\sum_{1\le a\ne b\le n_1} (h_{ab}^r)^2-\sum_{r=n+1}^{m}\sum_{1\le a\ne b\le n_1} h_{aa}^r h_{bb}^r=$$$$\overbrace{\sum_{r=n+1}^{m}\sum_{1\le a\ne b\le n_1} (h_{ab}^r)^2+\left(\sum_{r=n+1}^{m} \bigl((h_{11}^r)^2 + \cdots + (h_{n_1n_1}^r)^2\bigr)\right)}$$$$\overbrace{- \left(\sum_{r=n+1}^{m} \bigl((h_{11}^r)^2 + \cdots + (h_{n_1n_1}^r)^2\bigr)\right)-\sum_{r=n+1}^{m}\sum_{1\le a\ne b\le n_1} h_{aa}^r h_{bb}^r}.$$

By means of the binomial theorem, we deduce that 
$$\overbrace{\sum_{r=n+1}^{m}\sum_{1\le a\ne b\le n_1} (h_{ab}^r)^2+\left(\sum_{r=n+1}^{m} \bigl((h_{11}^r)^2 + \cdots + (h_{n_1n_1}^r)^2\bigr)\right)}=\sum_{r=n+1}^{m}\sum_{a, b=1}^{n_1}  (h_{ab}^r)^2,$$
and
$$\overbrace{- \left(\sum_{r=n+1}^{m} \bigl((h_{11}^r)^2 + \cdots + (h_{n_1n_1}^r)^2\bigr)\right)-\sum_{r=n+1}^{m}\sum_{1\le a\ne b\le n_1} h_{aa}^r h_{bb}^r}=$$$$- \sum_{r=n+1}^{m} (h_{11}^r+ \cdots  +h_{n_1n_1}^r)^2.$$

Next, by combining the last three equations together we obtain 
\begin{equation}\label{1070}
-\left(\sum_{r=n+1}^{m}\sum_{1\le a\ne b\le n_1} \bigl(h_{aa}^r h_{bb}^r- (h_{ab}^r)^2\bigr)\right)=\sum_{r=n+1}^{m}\sum_{a, b=1}^{n_1}  (h_{ab}^r)^2- \sum_{r=n+1}^{m} (h_{11}^r+ \cdots  +h_{n_1n_1}^r)^2.
\end{equation}

By Definition \ref{9}, the second term in the right hand side vanishes whenever $\varphi$ is $\mathcal{D}_T$-minimal, consequently (\ref{1070}) reduces to
\begin{equation}\label{1071}
-\left(\sum_{r=n+1}^{m}\sum_{1\le a\ne b\le n_1} \bigl(h_{aa}^r h_{bb}^r- (h_{ab}^r)^2\bigr)\right)=\sum_{r=n+1}^{m}\sum_{a, b=1}^{n_1}  (h_{ab}^r)^2.
\end{equation}
Combining (\ref{1071}) and (\ref{1072}), it yields to 
$$||h||^2= 2 \tilde \tau (T_xM^n) - 2\tilde \tau (T_x N_T)-2 \tilde\tau (T_x N_\perp)-2 \frac {n_2 \Delta f}{f}+ n^2||\vec H||^2$$$$+\sum_{r=n+1}^{m}\sum_{a, b=1}^{n_1}  (h_{ab}^r)^2$$$$-\left(\sum_{r=n+1}^{m}\sum_{n_1+1\le A\ne B\le n} \bigl(h_{AA}^r h_{BB}^r- (h_{AB}^r)^2\bigr)\right).$$

Equivalently, 
$$||h||^2\ge 2 \tilde \tau (T_xM^n) - 2\tilde \tau (T_x N_T)-2 \tilde\tau (T_x N_\perp)-2 \frac {n_2 \Delta f}{f}+ n^2||\vec H||^2$$$$-\left(\sum_{r=n+1}^{m}\sum_{n_1+1\le A\ne B\le n} \bigl(h_{AA}^r h_{BB}^r- (h_{AB}^r)^2\bigr)\right).$$

Again, by adding and subtracting similar term technique, the above inequality becomes
$$||h||^2\ge 2 \tilde \tau (T_xM^n) - 2\tilde \tau (T_x N_T)-2 \tilde\tau (T_x N_\perp)-2 \frac {n_2 \Delta f}{f}+ n^2||\vec H||^2$$$$
-\sum_{r=n+1}^{m}\left( (h_{n_1+1n_1+1}^r)^2 + \cdots + (h_{nn}^r)^2+\sum_{n_1+1\le A\ne B\le n} h_{AA}^r h_{BB}^r\right)$$$$
+\sum_{r=n+1}^{m}\left( (h_{n_1+1n_1+1}^r)^2 + \cdots + (h_{nn}^r)^2+\sum_{n_1+1\le A\ne B\le n} (h_{AB}^r)^2\right).$$

Applying the binomial theorem on the last two terms of the above equation, we derive that 
$$||h||^2\ge 2 \tilde \tau (T_xM^n) - 2\tilde \tau (T_x N_T)-2 \tilde\tau (T_x N_\perp)-2 \frac {n_2 \Delta f}{f}+ n^2||\vec H||^2$$$$
-\sum_{r=n+1}^{m} (h_{n_1+1n_1+1}^r+ \cdots  +h_{nn}^r)^2$$$$
+\sum_{r=n+1}^{m}\sum_{A, B=n_1+1}^{n} (h_{AB}^r)^2.$$

Consequently,
$$||h||^2\ge 2 \tilde \tau (T_xM^n) - 2\tilde \tau (T_x N_T)-2 \tilde\tau (T_x N_\perp)-2 \frac {n_2 \Delta f}{f}+ n^2||\vec H||^2$$$$
-\sum_{r=n+1}^{m} (h_{n_1+1n_1+1}^r+ \cdots  +h_{nn}^r)^2.$$

We know that the last term in the right hand side of the above inequality is equal to $-n^2||\vec H||^2$ for $\mathcal{D}_T$-minimal warped product $CR$-submanifolds. By this fact, the inequality of statement $(i)$ follows immediately from the above inequality. 

Now, the equality sign of the inequality in $(i)$ holds if and only if 
$$(a)~h(\mathcal{D}_T, \mathcal{D}_T)=0,~~~~~~~~~~~~(b)~h(\mathcal{D}_\perp, \mathcal{D}_\perp)=0.$$
Hence, we need to show that $(a)$ and $(b)$ hold if and only if $N_T$, $N_\perp$ and $M^n$ are respectively totally geodesic, totally umbilical and minimal submanifolds in $\tilde M^m$.

First, assume that $(a)$ and $(b)$ are satisfied. Since $M^n=N_T\times _fN_\perp$ is a warped product, then $N_T$ and $N_\perp$ are totally geodesic and totally umbilical in $M^n$, respectively. Therefore, part $(a)$ above implies that the first factor is a totally geodesic submanifold in $\tilde M^m$. The second factor is totally umbilical in $\tilde M^m$ because of part $(b)$. Moreover, (b) and (a) together imply that $M^n$ is minimal in $\tilde M^m$.

 For the converse, $(a)$ is clear. To obtain $(b)$, we first notice that minimality and $\mathcal{D}_T$-minimality of $M^n$ in $\tilde M^m$ yield to $\mathcal{D}_\perp$-minimality of $M^n$ in $\tilde M^m$. Hence, Lemma \ref{two} proves $(b)$. This gives the assertion. 
\end{proof}

\section{Special Inequalities and Applications}

As a first application, we embark on by deriving the three theorems of [\cite{2211gg}] from Theorem \ref{299} to be particular case theorems. For this, consider the warped product $CR$-submanifolds of type $N_T\times _fN_\perp$ in complex space forms. Since the ambient manifold $\tilde M^m$ of Theorem \ref{299} is an arbitrary Kahler manifold, we can consider $\tilde M^m$ to be a complex space form $\tilde M^{2m}(c_{Ka})$. Hence, for every $CR$-warped product $M^n=N_T\times _fN_\perp$ in $\tilde M^{2m}(c_{Ka})$, we just use the curvature tensor of complex space forms (\cite{2211gg}) to compute the following
$$2\biggl(\tilde \tau (T_xM^n)-\tilde \tau (T_xN_1)-\tilde \tau (T_xN_2)\biggr)=\frac{c_{Ka}}{4}\biggl(n(n-1)+3n_1-n_1(n_1-1)-3n_1-n_2(n_2-1)\biggr)$$$$=\frac{c_{Ka}n_1n_2}{2}.~~~~~~$$
Substituting the above expression in Theorem \ref{299}, because $CR$-warped product submanifolds of Kaehler manifolds are $\mathfrak{D_1}$-minimal, we obtain the following theorem as special case. 
\begin{theorem}\label{dp.}
Let $\varphi :M^n=N_T\times _fN_\perp \longrightarrow \tilde M^m$ be an isometric immersion of a warped product $CR$-submanifold $M^n$ into a complex space form $\tilde {M}^m$. Then, we have
 \begin{enumerate}
\item[(i)]$||h||^2\ge 2n_2\biggl(||\nabla (\ln f)||^2-\Delta (\ln f)+\frac{c_S+3}{2}s+1\biggr).$
\item[(ii)] The equality in (i) holds identically if and only if $N_T$, $N_\perp$ and $M^n$ are totally geodesic, totally umbilical and minimal submanifolds in $\tilde M^m$, respectively.
\end{enumerate}
\end{theorem}
  
\begin{remark}
Inequalities of Theorems 4.1, 5.1 and 6.1 in [\cite{2211gg}] are special cases of Theorem \ref{299}, where the ambient manifold is a complex Euclidean, a complex projective and a complex hyperbolic space, respectively.
\end{remark}

As another application of Theorem \ref{299}, we have
\begin{corollary}
Let $M^n=N_T\times _fN_\perp$ be a warped product $CR$-submanifold in a Kahler manifold $\tilde M^{m}$ and suppose $N_T$ is compact. Denote by $dv_T$ and $vol(N_T)$ the volume element and the volume on $N_T$. Let $\lambda_T$ be the first non zero eigenvalue of the Laplacian on $N_T$. Then
$$\frac{1}{2}\int_{N_T} ||h||^2 dv_T\ge n_1\biggl(\tilde\tau(T_xM)-\tilde\tau(T_xN_T)-\tilde\tau(T_xN_\perp)\biggr) vol(N_T)+n_1\lambda_T \int_{N_T} (\ln f)^2 dv_T.$$
\end{corollary}
\begin{proof}
From the minimum principle we have 
$$\int_{N_T} ||\nabla \ln f||^2 dv_T\ge \lambda_T\int_{N_T} (\ln f)^2 dv_T.$$
Now we have to integrate on $N_T$ the inequality of Theorem \ref{299} which is satisfied by the norm of $h$, and then we obtain immediately the result.
\end{proof}

Above integration over $N_T$ can be generalized to integration of a general measurable manifold with invariance properties. For this we will state the following:

\begin{theorem}
Let $M_{\mu}$ be a measurable manifold with a measure $\mu$ defined on it. Moreover, let $g:\mu \rightarrow \mu'$ be an invariance transformation from measure $\mu$ to measure $\mu'$. Then, we can express the integral $\int_{M_{\mu}} X$ over a quantity $X$ as the limit 
$\int_{M_{\mu}} = lim_{g \mapsto id} \sum_{x \in M_{\mu}} \mu(x) X(gx)$
where $x$ is an element of the manifold, here, the covering basis of it and $id$ is the identity operator. 
\end{theorem} 
\begin{proof}
Consider two values of a quantity $X$, namely $X(gx)$ and $X(x)$ for any manifold covering $x$. The transformation $g$ will now tend to the identity transform. Thus, $X(gx)-X(x)$ will be infinitesimal in the case when the function is smooth.  
In non-smooth case, the transformation $g$ will shift the covering $x$ from the singularity apart by appropriate choice of it. Since the manifold $M_{\mu}$ is measurable, we can define a measure on it and can also compute a measure-weighted sum over $X$.
\end{proof}

\section{An Extension of the Inequality to Warped Product $CR$-Submanifolds in Nearly Kahler Manifolds}

\begin{theorem}\label{}
Let $\varphi :M^n=N_T\times _fN_\perp \longrightarrow \tilde M^m$ be an isometric immersion of a warped product $CR$-submanifold $M^n$ into a nearly Kahler manifold $\tilde {M}^m$. Then, we have
 \begin{enumerate}
\item[(i)]$||h||^2\ge 2n_2\biggl( \frac{c_{Ke}-3}{2}s-\Delta (\ln f)\biggr).$
\item[(ii)] The equality in (i) holds identically if and only if $N_T$, $N_\perp$ and $M^n$ are totally geodesic, totally umbilical and minimal submanifolds in $\tilde M^m$, respectively.
\end{enumerate}
\end{theorem}

Following a similar analogue of the previous section, we can use the above theorem to obtain a special inequality of generalized complex space forms 
\begin{theorem}\label{}
Let $\varphi :M^n=N_T\times _fN_\perp \longrightarrow \tilde M^m$ be an isometric immersion of a warped product $CR$-submanifold $M^n$ into a generalized complex space form $\tilde {M}^m$. Then, we have
 \begin{enumerate}
\item[(i)]$||h||^2\ge 2n_2 \biggl(||\nabla (\ln f)||^2-\Delta (\ln f)+n_1\frac{c_{RK}+3\gamma}{4}\biggr).$
\item[(ii)] The equality in (i) holds identically if and only if $N_T$, $N_\perp$ and $M^n$ are totally geodesic, totally umbilical and minimal submanifolds in $\tilde M^m$, respectively.
\end{enumerate}
\end{theorem}

It is clear that the above theorem generalizes Theorem \ref{dp.}. To see that, just let $\gamma$ vanish.

\section{Research problems based on the main Inequality: Theorem \ref{299}}

Due to the results of this paper, we hypothesize a pair of open problems, the first is about proving this inequality whereas the second is to classify warped products $CR$-submanifolds. 

Firstly, since warped product $CR$-submanifolds do exist if the ambient manifold is locally conformal Kahler space form, we suggest the following
\begin{problem}\label{}
Prove the above inequality for warped product $CR$-submanifolds in locally conformal Kahler space forms.
\end{problem}

Secondly, we asked
\begin{problem}\label{}
Can we classify warped product $CR$-submanifolds satisfying the equality cases of this inequality in locally conformal Kahler space forms ?
\end{problem}

\begin{center}

Acknowledgment 

The first author want to offer many thanks for his university, PTUK, Palestine Technical University- Kadoori
\end{center}


\begin{thebibliography}{99}

\bibitem{A}\label{dD}K. Arslan, R. Ezentas, I. Mihai and G. Murathan, {\it Contact CR-warped product submanifolds in Kenmotsu space forms}, J. Korean Math. Soc., 42(5)(2005),
1101-1110.

\bibitem{pom} Bejancu, A. (1986). Geometry of CR-submanifolds. {\it D. Reidel Publishing Company}.

\bibitem{ddyy7} Bishop, R. L. \& O'Neill, B. (1969). Manifolds of negative curvature. {\it Transactions of the American Mathematical Society}, 145: 1-49.

\bibitem{B}\label{Q}D. E. Blair, {\it Almost contact manifolds with Killing structure tensors I}, Pacific J. Math.,
39 (1971), 285-292.

\bibitem{2233ee} Chen, B. Y. (2002). Geometry of warped products as Riemannian submanifolds and related problems. {\it Soochow Journal of Mathematics}, 28: 125-156.

\bibitem{6677bb} Chen, B. Y. (2002). On isometric minimal immersions from warped products into real space forms. {\it Proceedings of the Edinburgh Mathematical Society}, 45: 579-587.

\bibitem{2211gg} Chen, B. Y. (2003). Another general inequality for CR-warped products in complex space forms. {\it Hokkaido Mathematical Journal}, 32(2): 415-444.

\bibitem{yyhh88} Chen, B. Y. (2005). On warped product immersions. {\it Journal of Geometry}, 82(1-2): 36-49.

\bibitem{55kk99} Chen, B. Y. (2008). {\it $\delta$-invariants, inequalities of submanifolds and their applications: in Topics in differential Geometry}. Editura Academiei Rom$\hat a$ne, Bucharest, 29-155.

\bibitem{ssee44} Chen, B. Y. (2013). A survey on geometry of warped product submanifolds. {\it Journal of Advanced Mathematical Studies}, 6(2): 1-43. arXiv:1307.0236v1 [math.DG].

\bibitem{kkxx77} Kim, J. S., Liu, X. \& Tripathi, M. M. (2004). On semi-invariant submanifolds of nearly trans-Sasakian manifolds. {\it International Journal of Pure and Applied Mathematics}, 1: 15-34.

\bibitem{H}\label{S}I. Hasegawa and I. Mihai, {\it Contact CR-warped product submanifolds in Sasakian manifolds},
Geom. Dedicata, 102 (2003), 143-150.

\bibitem{K}\label{a}K. Kenmotsu, {\it A class of almost contact Riemannian manifolds}, Tohoku Math. J., 24
(1972), 93-103.

\bibitem{K}\label{L}V. A. Khan, K. A. Khan and S. Uddin, {\it Contact CR-warped product submanifolds of Kenmotsu
manifolds}, Thai Journal of Mathematics, vol. 6, no. 1, pp. 138–145, 2008.

\bibitem{Mi}\label{T} I. Mihai, {\it Contact CR-warped product submanifolds in Sasakian space forms}, Geom.
Dedicata, 109 (2004), 165-173.

\bibitem{M}\label{I}M.-I. Munteanu, {\it Warped product contact CR-submanifolds of Sasakian space forms}, Publ.
Math. Debrecen 66 (2005), no. 1-2, 75–120.

\bibitem{99kkrr} Mustafa, A., Uddin, S., Khan, V. A. \& Wong, B. R. (2013). Contact CR-warped product submanifolds of nearly trans-Sasakian manifolds. {\it Taiwanese Journal of Mathematics}, 17(4): 1473-1486.

\bibitem{11bbyy} Mustafa, A., Uddin, S. \& Wong, B. R. (2014). Generalized inequalities on warped product submanifolds in nearly trans-Sasakian manifolds. {\it Journal of Inequalities and Applications}, 2014: 346.

\bibitem{aassll} Mustafa, A., De, A. \& Uddin, S. (2015). Characterization of warped product submanifolds in Kenmotsu manifolds. {\it Balkan Journal of Geometry and Its Applications}, 20(1): 86-97.

\bibitem{iijj77} O'Neill, B. (1983). {\it Semi-Riemannian geometry with applictions to relativity.} New York: Academic Press.

\bibitem{S}\label{z} S. Sasaki, {\it On differentiable manifolds with certain structures which are closely related to almost contact}
structure Tohoku Math. J. 12(1960) 459–76.

\bibitem{obl}A. Bejancu, {\it Oblique warped products}, Journal of Geometry and Physics, 57(3), (2007), pp. 1055-1073.

\bibitem{seq}, U. Chand Dea, S. Shenawyb, B. Unal, {\it Sequential Warped Products: Curvature and Conformal Vector Fields}, Filomat, 33(13), (2019), 40714083.

\bibitem{qESeq}F. Karaca, C.  \"{O}zg \"{u}r, {\it On quasi-Einstein sequential warped product manifolds}, Journal of Geometry and Physics, 165, (2021), 104248.

\bibitem{pndp}A. Pigazzini, C. \"{O}zel, P. Linker and S. Jafari, {\it On PNDP-manifold}, Poincare J. anal. appl., 8(1(I)), (2021), pp.111-125.





\end{thebibliography}
\end{document}